\documentclass{article}
\usepackage{amsmath}
\usepackage{amsthm}
\usepackage{amssymb}
\usepackage{dsfont}
\usepackage{calc}
\usepackage{eufrak} 

\usepackage{authblk}
\usepackage{theoremref}
\usepackage{verbatim}

\begin{document}
\bibliographystyle{alpha}
\newtheorem{theorem}{Theorem}
\newtheorem{lemma}{Lemma}
\newtheorem{definition}{Definition}
\newtheorem{proposition}{Proposition}
\newtheorem{remark}{Remark}
\newtheorem{property}{Property}
\newtheorem{corollary}{Corollary}
\newcounter{casenum}
\newenvironment{caseof}{\setcounter{casenum}{1}}{\vskip.5\baselineskip}
\newcommand{\case}[2]{\vskip.5\baselineskip\par\noindent {\bfseries Case \arabic{casenum}:} #1\\#2\addtocounter{casenum}{1}}

\title{Group Embeddings with Algorithmic Properties}
\author{Arman Darbinyan}
\date{}
\affil{\emph{}}
\maketitle

\let\thefootnote\relax\footnote{
{\bf  2010 Math.\ Subject Classification}. 20F05, 20F10
 \par

}

\section{Abstract}
We show that every countable group $H$ with solvable word problem can be subnormally embedded into a 2-generated group $G$ 
which also has solvable word problem. Moreover, the membership problem for $H<G$ is also solvable.  We also give estimates of time and space complexity of the word problem in $G$ and of the membership problem for $H<G$.
\section{Introduction}

In the famous paper \cite{HNN} by Higman, B.H.Neumann and H. Neumann in 1949,  using constructions based on HNN-extensions, it was shown that every countable group can be embedded in a group generated by two elements. Later, B.H. Neumann and H. Neumann suggested an alternative embedding construction based on wreath products \cite{NN}, which allowed them to show that every countable solvable group can be embedded in a 2-generated solvable group. Further development of these ideas was done by Hall \cite{hall}. Subsequently, other constructions based on this ideas were introduced, 
where the 2-generated group inherited some other properties of the initial group. For example, embeddings of periodic groups, Phillips \cite{phillips}; property of residually finiteness, Wilson \cite{wilson}; SD-groups, subnormal and verbal embeddings, orderable groups, Mikaelian \cite{mikaelian sd} \cite{mikaelian subnormal} \cite{mikaelian verbalian} \cite{mikaelian order}; embedding which preserves elementary amenability, Olshanskii, Osin \cite{Ol'sh Osin}, etc.
Most of the aforementioned embedding constructions were motivated either by the desire to better control the algebraic structure or geometric properties of the embedding, but none of the constructions based on wreath products was concerned about algorithmic properties of the resulting group.\\

Algorithmic properties for embeddings of countable groups were investigated by Clapham \cite{Clapham}, Collins \cite{Collins}, Thompson \cite{Thompson}, Boone, Higman \cite{Boone}, Birget, Olshanskii, Rips, Sapir \cite{Birget}, Miller III \cite{miller}, Olshanskii, Sapir \cite{Ol'sh Sapir} and others. All the embeddings in these papers use constructions based on HNN-extensions, which have a disadvantage of leading to "large" finitely generated groups even if the initial group is relatively "small" countable group. For example, the resulting finitely generated group always contains nonabelian free subgroups.\\
~\\
In this paper we will investigate algorithmically well behaving embeddings based on wreath products and using ideas which can be traced back to B.H. Neumann, H.Neumann \cite{NN} and Hall \cite{hall}. \\
~\\

For a given alphabet (i.e., set of letters) $X$, we denote by $X^*$ the set of words in the alphabet $X \cup X^{-1}$. And for a given word $u\in X^*$, we denote its length with respect to this alphabet by $\left\|u\right\|_X$ or, if there will not occur ambiguity, we will simply denote it by $\left\|u\right\|$.\\
~\\
Let $S=\{ a^{(1)}, a^{(2)}, \ldots \}$ be a generating set of a countably generated group $H$. Since Turing machines work only with finite input alphabets, for considering word problem in $H$ we will encode the elements of $S$ by a quaternary code in the following way: encode $a^{(i)}$ as $3b_i$ and $(a^{(i)})^{-1}$ as $4b_i$, where $b_i$ is the binary presentation of the index $i$. Clearly, the code of $(a^{(i)})^{\pm 1}$ is not longer than $2+\left\lfloor log_2(i) \right\rfloor$ (we define $log_2(n)=0$, for $n \leq 0$). Hence, for the alphabet $S_n=\{ a^{(1)}, a^{(2)}, \ldots, a^{(n)} \}$, if $v \in S_n^{*}$, then the length of the quaternary code of $v$, which we will denote by $\mathfrak{L}(v)$, is less than or equal to 
\begin{equation}
\label{eq 0000}
\left\| v \right\|_{S_n} (2+\left\lfloor log_2(n) \right\rfloor)=
\left\| v \right\|_{S} (2+\left\lfloor log_2(n) \right\rfloor).
\end{equation}
For example, the code of $a^{(1)}$ is $31$, $\mathfrak{L}(a^{(1)}) = 2+\left\lfloor log_2(1) \right\rfloor=2$, the code of
$(a^{(1)})^{-1}$ is $41$, $\mathfrak{L}((a^{(1)})^{-1}) = 2+\left\lfloor log_2(1) \right\rfloor=2$.\\

 We will assume that the word problem for the group $H$ is solvable, in the sense that there exists an algorithm, such that for every encoded word from $S^*$ it decides whether the corresponding element in $H$ is trivial or not.\\
~\\
For convenience we introduce the notation $Lg(n)=2+\left\lfloor  log_2(n) \right\rfloor$. Also, we define the commutator of $x$ and $y$ as $[x,y]=xyx^{-1}y^{-1}$. To estimate complexity functions for the below described algorithms we will use standard asymptotic notations. Namely, $ F(n)=O(f(n))$ if there exists a real number $n_0$ and a positive constant $M$, such that $n>n_0$ implies $|F(n)|\leq M|f(n)|$. 
\\ 
~\\
Our main result is the following
\begin{theorem}
\label{th 1}
Let $ H $ be a group generated by a countable set $S=\{ a^{(1)}, a^{(2)}, \ldots\}$ and having solvable word problem with respect to the above described encoding of $S$.
Then there is a subnormal embedding $\phi$ of the group $H$ into a two generator group $G=gp \langle c, s \rangle$ with solvable word problem, and it can be recognized whether a word from $\{ c, s \}^*$ represents an element of $H$ or not, i.e., the membership problem for $H < G$ is solvable as well.
Moreover, given a word $w$ from $\{ c, s\}^*$, if $f, g : \mathbb{N} \rightarrow \mathbb{N}$ are time complexity and space complexity functions for the word problem in $H$ respectively, where we consider the word problem with respect to the quaternary encoding described above, then\\
$(i)$ the word problem in $G$ can be solved in time
\[O(n^3 Lg(n) + n^2 f( n Lg(Lg(n)))), \]
where $n$ is the length of the word $w$ from $\{c, s\}^*$.
\\ 

$(ii)$ the membership problem for $H < G$ can be solved in time 
\[O( n^4 Lg(n) + n^3 f( nLg(Lg(n)))). \]\\
$(iii)$ the space complexity for the word (resp. membership) problem in $G$ (resp. for $H<G$) is 
$O(g(n)Lg(Lg(n)))$.

\end{theorem}
~\\
Roughly speaking, the space function of an algorithm is a measurement which shows the maximal amount of memory (or space) used by a (Turing) machine at some point of time, while running the algorithm and time complexity is the time required by the machine to run the algorithm, and the domain of these functions shows the length of the input data. For more rigorous definitions of time and space complexity functions we refer to the articles \cite{olshanskii space functions} \cite{sapir birget rips}.\\
~\\
Directly from the formulas in $(i)$ and in $(ii)$ we obtain the following
\begin{corollary}
\label{corollary 0}
If the word problem for $H$ is solvable in polynomial time, then the word (resp. membership) problem for $G$ (resp. for $H<G$) is solvable in polynomial time as well.
\end{corollary}
In other words, Corollary \ref{corollary 0} means that, if the word problem in $H$ belong to the class $P$, then the word problem in $G$ belongs to the class $P$ too. In the proof of Theorem \ref{th 1} it is shown that if the word problem in $H$ (resp. the membership problem for $H<G$) belongs to the class $NP$, then the word problem in $G$ (resp. the membership problem for $H<G$) belongs to the class $NP$ too. Also, from the formula in $(iii)$ it follows that if the word problem in $H$ belongs to the space complexity class $L$ or $NL$, then the word problem in $G$ and the membership problem for $H<G$ belong to the same space complexity class.\\
~\\
In the statement of Theorem \ref{th 1},  we said that $\phi$ is a subnormal embedding, meaning that there exists a finite chain $G_0, G_1 \ldots, G_k$ of subgroups of $G$, such that $\phi(H)=G_0 \triangleleft G_1 \triangleleft \ldots \triangleleft G_k=G$.\\
~\\
From the embedding construction given in the proof of Theorem 1, we also obtain
\begin{corollary}
\label{corollary 1}
If $H$ is a solvable group of length $l$, then $G$ is a solvable group of length $l+2$.
\end{corollary}

A group $G$ is called locally indicable if for every finitely generated subgroup $1 \neq H \leq G$, there exists a surjective homomorphism from $H$ onto the infinite cyclic group $\mathbb{Z}$.

\begin{corollary}
\label{corollary 2}
If $H$ is torsion free or locally indicable, then so is $G$.
\end{corollary}

Before moving forward to prove the theorem, let us recall the definition of wreath products. Given two groups $A$ and $B$, the base subgroup $A^{B}$ is the set of functions from $B$ to $A$ with pointwise multiplication and the group $B$ acts on $A^B$ from the left by automorphisms, such that for $f\in A^B$ and $b \in B$, the resulting function $bf$ is given by 
$$ (b f)(x) = f(xb), \forall x \in B. $$
For the convenience we will denote $bf$ by $f^b$.

\begin{definition}
The wreath product $A Wr B$ is defined as a semidirect product $ A^B \ltimes B$, with the multiplication $(f_1b_1)(f_2 b_2) = f_1f_2^{b_1} b_1b_2$ for all $ f_1, f_2 \in A^B, ~b_1, b_2 \in B$. 
\end{definition}
Note that from the above definition it follows that $f^b=bfb^{-1}$. Also, the inverse of $fb$ is $(f^{-1})^{b^{-1}}b^{-1}$, the conjugate of 
$f_1b_1$ by $f_2b_2$ is $(f_1b_1)^{f_2b_2}=(f_2 f_1^{b_2} (f_2^{-1})^{b_2b_1b_2^{-1}})b_2b_1b_2^{-1}$.

\begin{definition}
For $g=fb\in A Wr B$ we call $f$ the passive and $b$ the active part of the group element $g$.
\end{definition}

The group $G$ is called linearly ordered (or fully ordered) if there is a linear order $<$ defined on the elements of $G$, such that
for any elements $g_1$ and $g_2$ from $G$, if $g_1 \leq g_2$, then $g_1x \leq g_2x$ and $xg_1 \leq xg_2$ for all $x \in G$.\\

In the paper \cite{mikaelian order} (see also \cite{Darbin Mikael}) the following fact was proved
\begin{lemma}
\label{lem darbinyan mikaelian}
Let A and B be linearly ordered groups and $\Gamma$ be a subgroup of $ A Wr B$. If for
each $f b \in \Gamma$, $supp(f)$ is well-ordered, then $\Gamma$  is linearly orderable.

\end{lemma}

In what follows, this lemma applied to the group $G$ from Theorem 1 gives the following
\begin{corollary}
\label{corollary 3}
If $H$ is linearly orderable, then so is $G$.
\end{corollary}


\section{Proof of Theorem \ref{th 1}}
In the proof, first we will describe the construction of the group $G$ and the embedding $\phi: H \rightarrow G$, then we will
describe the algorithms for the word problem and the membership problem, simultaneously estimating the time complexity of these algorithms.
Finally, we will turn to the proof of the part (iii).\\

Let us consider the group $H Wr Z$, where $Z=\langle z \rangle$ is an infinite cyclic group.\\
Take $b^{(i)} \in H^Z$, such that 
\[ b^{(i)}(z^k) = \left\{
                      \begin{array}{ll}
                       a^{(i)} & \mbox{if $k > 0$ ,}\\
                        1& \mbox{otherwise.}
                     \end{array}
                    \right. , i= 1,2, \ldots
\]
Note that, 
\[
(zb^{(i)}z^{-1})(z^k) = \left\{
                      \begin{array}{ll}
                       a^{(i)} & \mbox{if $k \geq 0$ ,}\\
                        1& \mbox{otherwise.}
                     \end{array}
                    \right.
\]
(this is because $zb^{(i)}z^{-1} (z^{k}) = (zb^{(i)})(z^{k}) = b^{(i)}(z^{k+1}))$.\\
Therefore, 
\[
[z, b^{(i)}](1) = (zb^{(i)}z^{-1})(1) (b^{(i)^{-1}})(1)=a^{(i)}
\]
and $[z, b^{(i)}](z^k)=1$ for all $k \neq 0,$ 
i.e., the functions $[z, b^{(i)}]$
generate a subgroup isomorphic to H. So the mapping $a^{(i)} \mapsto [z, b^{(i)}]$
embeds the group $H$ into the subgroup $K =gp\langle z, b^{(i)} | i \in \mathbb{N} \rangle$ of the wreath product.
Now let us consider the group $ K Wr S$, where $ S = \langle s \rangle $ is an infinite cyclic group.\\
Consider $c \in K^{S}$, such that
 \[           c(s^k) = \left\{ 
                       \begin{array}{lll}
                       z & \mbox{if $k = 1$ ,}\\
                       b^{(i)} & \mbox{if $i > 0$ and $ k = 2^i$, }\\
											1& \mbox{otherwise.}
                     \end{array}
                    \right.
\]

For our construction it is vital that the support of $c$ is sparse. One of the reasons can be seen in the next described key property.\\
 
The element $[c,c^{s^{2^i-1}}]$, considered as a $ S \rightarrow	K$ function, takes a non-trivial value $a^{(i)}$ (more precisely, $[z,b^{(i)}]$) only at $s$. Indeed, 
\[
[c,c^{s^{2^i-1}}](s) = c(s) (c^{s^{2^i-1}})(s) c^{-1}(s) c^{-s^{2^i-1}}(s)
=c(s) c(s^{2^i}) (c(s))^{-1} (c(s^{2^i}))^{-1}
\]
\[
=z b^{(i)} z^{-1} b^{(i)^{-1}}=[z, b^{(i)}],
\]
and in general, if $[c,c^{s^{2^i-1}}](s^k) \neq 1$, then $c(s^k) \neq 1$, $c^{s^{2^i-1}}(s^k)=c(s^{2^i+k-1}) \neq 1$. But this means
that $k$ and $2^i+k-1$ are some powers of $2$, which can take place only if $k=1$.

This property of $[c, c^{2^i-1}]$ implies that the map $\phi: H \rightarrow G=gp\left\langle  c, s \right\rangle$, given by $\phi: a^{(i)} \mapsto [c, c^{2^i-1}]$, embeds $H$ into the 2-generated group $G$.

Now let us describe a procedure which recognizes identities in $G$ among the words of $\{ c, s \}^*$.\\
Let 
\begin{equation}
\label{eq 1}
 w = s^{\alpha_0}c^{\beta_1} \ldots s^{\alpha_{n-1} }c^{\beta_n} s^{\alpha_n}
\end{equation}
 be an arbitrary word in $\{c, s\}^*$, where $\alpha_0, \alpha_1, \ldots, \alpha_n, \beta_1, \ldots, \beta_n $ are some non-zero integers with a possible exception of $\alpha_0$ and $\alpha_n$ (they could be equal to $0$). Particularly, $\left\|w \right\| _ {\{ c, s\}} \geq 2n-1$.

Note that $w$ is equal to
\begin{equation}
\label{eq 2}
(c^{s^{\gamma_1}})^{\beta_1} (c^{s^{\gamma_2}})^{\beta_2} \ldots  (c^{s^{\gamma_n}})^{\beta_n}s^{\gamma}  
\end{equation}
 in $G$, where 
\begin{equation}
\label{eq 0}
\gamma=\sum_{i=0}^{n}{\alpha_i}, ~\gamma_j = \sum_{i=0}^{j-1}{\alpha_i},
\end{equation}
thanks to the identity
$s^{\alpha} c^{\beta} =  (c^{s^{\alpha}})^{\beta}s^{\alpha}$. Also, note that 
\[
\left\| w \right\|_{\{c,s\}}=\sum_{i=0}^{n}{|\alpha_i|}+\sum_{i=1}^{n}{|\beta_i|} \geq |\gamma|+\sum_{i=1}^{n}{|\beta_i|}.
\]

A direct consequence of these formulas is that for passing from the form (\ref{eq 1}) to the form (\ref{eq 2}) we only need to find $\gamma$-s using the above described formula (\ref{eq 0}). Since the binary lengths of $\gamma$-s are not larger than $log_2(\left\| w\right\| )$, we get that $\sum_{i=0}^{j-1}{\alpha_i}$ can be computed in  $O(j \cdot log_2(\left\| w\right\|))$ time. This means that the whole computation in (\ref{eq 0}) can be done in $O(log_2(\left\| w\right\|)n^2)$ time. Also, we have 
\begin{equation}
\label{eq 00}
O(log_2(\left\| w\right\|)n^2) \leq O(log_2(\left\| w\right\|)\left\| w \right\|^2).
\end{equation}
~\\
For the presentation (\ref{eq 2}) of $w$, let us define the set $B_i = \{ j \in \{ 1, \ldots, n\} ~|~ \gamma_j = \gamma_i \}$ for $ i = 1, \ldots, n$. 

\begin{lemma}

\label{lem 1}
If $(c^{s^{\gamma_1}})^{\beta_1}(c^{s^{\gamma_2}})^{\beta_2} \ldots (c^{s^{\gamma_n}})^{\beta_n}s^{\gamma} = 1$, then $ \sum_{j \in B_i}{\beta_j}=0$, for $i=1,\dots, n $ and $\gamma = 0$. 
\end{lemma}

\begin{proof}

Denote $(c^{s^{\gamma_1}})^{\beta_1}(c^{s^{\gamma_2}})^{\beta_2} \ldots (c^{s^{\gamma_n}})^{\beta_n}$ by $f$. The active part of $f(s^{-\gamma_i+1})$, regarded as an element of $K$, is equal to $z^{\sum_{j \in B_i}{\beta_j}}$ for $i= 1, \ldots, n$, because 
\[f(s^{-\gamma_i+1})= (c^{s^{\gamma_1}})^{\beta_1}(c^{s^{\gamma_2}})^{\beta_2} \ldots (c^{s^{\gamma_n}})^{\beta_n} (s^{-\gamma_i+1})\]
\[= [(c^{s^{\gamma_1}})^{\beta_1}(s^{-\gamma_i+1})][(c^{s^{\gamma_2}})^{\beta_2}(s^{-\gamma_i+1})] \ldots [(c^{s^{\gamma_n}})^{\beta_n}(s^{-\gamma_i+1})]\]
 and  
$(c^{s^{\gamma_j}})^{\beta_j}(s^{-\gamma_i+1})=(c(s))^{\beta_j}=z^{\beta_j}$ if $\gamma_j=\gamma_i$ and $(c^{s^{\gamma_j}})^{\beta_j}(s^{-\gamma_i+1})=(c(s^{\gamma_j-\gamma_i+1}))^{\beta_j} \in gp \langle b^{(i)} ~|~i \in \mathbb{N} \rangle$ if $\gamma_j \neq \gamma_i$.

In the same way as we obtained the rewriting (\ref{eq 2}), we can present $f(s^{-\gamma_i+1})$ in the form 
\[
((b^{(i_1)})^{z^{\eta_1}})^{\xi_1} ((b^{(i_2)})^{z^{\eta_2}})^{\xi_2} \ldots  ((b^{(i_m)})^{z^{\eta_m}})^{\xi_m}z^{\epsilon},
\]
for some integers $\eta_k, \xi_k, \epsilon$. By the analogy with the equations (\ref{eq 0}), we have $\epsilon=\sum_{j\in B_i}{\beta_j}$.
Hence, if $f(s^{-\gamma_i+1})=1$, then $\sum_{j\in B_i}{\beta_j}$ should be $0$. Also, trivially $\gamma$ should be $0$ too.

\end{proof}
~\\
Further we will use the notation 
$\gamma_0 =  max \{ |\gamma_1|, \ldots, |\gamma_n| \} $. 

\begin{lemma}
\label{lem 2}

If $|\mu|>3\gamma_0$ and $\sum_{j\in B_i}{\beta_j}=0,  i=1, \ldots, n$, then 
\[(c^{s^{\gamma_1}})^{\beta_1}(c^{s^{\gamma_2}})^{\beta_2} \ldots (c^{s^{\gamma_n}})^{\beta_n} (s^{\mu})=1.\]

\end{lemma}

\begin{proof}
Assume $f(s^{\mu})=(c^{s^{\gamma_1}})^{\beta_1}(c^{s^{\gamma_2}})^{\beta_2} \ldots (c^{s^{\gamma_n}})^{\beta_n}(s^{\mu}) \neq 1$.\\
First, we will show that there exist integers $k$ and $l$, $ 1 \leq k \neq l \leq n$, such that $\gamma_k \neq \gamma_l$ and $(c^{s^{\gamma_k}})^{\beta_k}(s^{\mu}) \neq 1$, $(c^{s^{\gamma_l}})^{\beta_l}(s^{\mu}) \neq 1$. Indeed, if this does not take place, then for a fixed $k$, $(c^{s^{\gamma_k}})^{\beta_k}(s^{\mu}) \neq 1$ implies that for arbitrary $l \notin B_k$ (i.e., for $\gamma_l \neq \gamma_k$) we have $(c^{s^{\gamma_l}})^{\beta_l}(s^{\mu}) = 1$. But this means that
 $f(s^{\mu}) = \prod_{i_j \in B_k}{(c^{s^{\gamma_{i_j}}})^{\beta_{i_j}}(s^{\mu})} = (c^{s^{\gamma_k}})^{\sum_{i_j \in B_k}{\beta_{i_j}}}(s^{\mu})=(c^{s^{\gamma_k}})^0(s^{\mu})=1$, a contradiction.

Thus, there exist $k$ and $l$, such that $1 \leq k \neq l \leq n$, $\gamma_k \neq \gamma_l$, $~(c^{s^{\gamma_k}})^{\beta_k}(s^{\mu})\neq 1$ and $(c^{s^{\gamma_l}})^{\beta_l}(s^{\mu})\neq 1$. 
But for any $\gamma \in \mathbb{Z}, (c^{s^{\gamma}})^{\beta}(s^{\mu})=(c(s^{\mu+\gamma}))^{\beta} \neq 1$ only if $\mu+\gamma=2^x$ i.e., $\mu = 2^x-\gamma$ for some nonnegative integer $x$.
Therefore, there exist nonnegative integers $x_1$ and $x_2$, such that $ x_1 \neq x_2 $ and $\mu = 2^{x_1}-\gamma_k = 2^{x_2}-\gamma_l$. Assume $x_1 > x_2$, then $2^{x_2}\leq 2^{x_1}-2^{x_2}= \gamma_k - \gamma_l \leq 2\gamma_0$. Hence
$|\mu|= |2^{x_2}-\gamma_l| \leq 2^{x_2}+|\gamma_l| \leq 2\gamma_0+\gamma_0=3\gamma_0$, a contradiction.

\end{proof}
~\\
 
A direct consequence of Lemma \ref{lem 1} and Lemma \ref{lem 2} is the following

\begin{lemma}
\label{lem 4444}
The word
\[ w = (c^{s^{\gamma_1}})^{\beta_1}(c^{s^{\gamma_2}})^{\beta_2} \ldots (c^{s^{\gamma_n}})^{\beta_n}s^{\gamma}\]
is trivial in $G$ if and only if\\
\begin{equation}
\label{eq 3}
  \gamma = 0 \mbox{~and} \sum_{j\in B_i}{\beta_j}=0 \mbox{~for~} i=1, \ldots, n
\end{equation}
 and \\
\begin{equation}
\label{eq 4}
(c^{s^{\gamma_1}})^{\beta_1}(c^{s^{\gamma_2}})^{\beta_2} \ldots (c^{s^{\gamma_n}})^{\beta_n}(s^{\mu})
\end{equation}
 is trivial in $G$ for all $-3\gamma_0 \leq \mu \leq 3\gamma_0$. 
\end{lemma}

Note that since $\gamma_0 \leq \left\| w \right\|$, in Lemma \ref{lem 4444} $-3\gamma_0 \leq \mu \leq 3\gamma_0$ can be replaced by $-3\left\| w \right\| \leq \mu \leq 3\left\|w \right\|$.\\
For the further we denote the element in (\ref{eq 4}) by $f_{\mu}$.\\
~\\
Exactly in analogy with the equation (\ref{eq 00}), the condition (\ref{eq 3}) can be checked in time
\begin{equation}
\label{eq new 1}
O(log_2(\left\|w\right\|)\left\| w \right\|^2).
\end{equation}

~\\
~\\
We have 
\begin{equation}
\label{eq 5}
f_{\mu}
 = [(c^{s^{\gamma_1}})^{\beta_1}(s^{\mu})][(c^{s^{\gamma_2}})^{\beta_2}(s^{\mu})] \ldots [(c^{s^{\gamma_n}})^{\beta_n}(s^{\mu})]
\end{equation}
and 
\begin{equation}
\label{eq 55}
(c^{s^{\gamma_i}})(s^{\mu})=
                    \left\{
                      \begin{array}{lll}
                       b^{(log_2(\gamma_i+\mu) )} & \mbox{if $\gamma_i+\mu$ is a natural power of 2,}\\
                       z & \mbox{if $\gamma_i+\mu=1$,}\\
												1& \mbox{otherwise.}
                     \end{array}
                    \right. 
\end{equation}
Thus,

in the case $\gamma_0 + \mu \leq 0$, we have $f_{\mu}=1$, otherwise, $f_{\mu}$ is an element in $gp \langle 1, z, b^{(j)} ~|~ 1\leq j \leq \left\lfloor log_2(\mu+\gamma_0)\right\rfloor\rangle $, whence can be presented in the form
\begin{equation}
\label{eq 56}
	z^{\zeta_0}(b^{(i_1)})^{\xi_1}z^{\zeta_1} (b^{(i_2)})^{\xi_2} \ldots  (b^{(i_m)})^{\xi_m}z^{\zeta_m},
\end{equation}
where $\zeta_j$-s are nonnegative and $\xi_j$-s are positive integers, $0<i_j \leq \left\lfloor log_2(\mu + \gamma_0)\right\rfloor$ and $m\leq n\leq \left\|w\right\|$. Also, 
$\zeta_j$-s, $i_j$-s and $\xi_j$-s can be calculated based on (\ref{eq 55}). This calculation can be done in time
$
O(n (log_2(\gamma_0 + \mu)) ) \leq O( \left\| w \right\| \left\lfloor log_2(\gamma_0 + \mu) \right\rfloor ). 
$
Since we are interested only in the case $-3\gamma_0 \leq \mu \leq 3\gamma_0$, for this case we have $\left\lfloor log_2(\gamma_0+\mu)\right\rfloor \leq \left\lfloor log_2(\gamma_0+3\gamma_0)\right\rfloor \leq 
 Lg(\left\| w \right\|).$ The following inequality put the just calculated time estimation in a more convenient form for our purposes
\begin{equation}
\label{eq 57}
O( \left\| w \right\| \left\lfloor log_2(\gamma_0 + \mu) \right\rfloor )\leq O(\left\| w \right\| Lg(\left\| w \right\|))
\end{equation}

After applying a transformation similar to the one, which led from (\ref{eq 1}) to (\ref{eq 2}), but this time based on the identity $z^{\zeta}(b^{(i)})^{\xi}=((b^{(i)})^{z^{\zeta}})^{\xi}z^{\zeta}$, (\ref{eq 5}) can be rewritten as
\begin{equation}
\label{eq 6}
((b^{(i_1)})^{z^{\eta_1}})^{\xi_1} ((b^{(i_2)})^{z^{\eta_2}})^{\xi_2} \ldots  ((b^{(i_m)})^{z^{\eta_m}})^{\xi_m}z^{\eta},
\end{equation}
where in analogy with the equation (\ref{eq 0}), we have
\begin{equation}
\label{eq 99}
\eta = \sum_{i=0}^m \zeta_i, ~ \eta_j = \sum_{i=0}^{j-1} \zeta_i.
\end{equation}
 
Using this formula, as in the case of equation (\ref{eq 00}), we obtain that if the exponential coefficients of (\ref{eq 56}) are given, the (exponential) coefficients of (\ref{eq 6}) can be calculated in time $O(log_2( \left\| w \right\| ) \left\| w \right\|^2).$ Hence, by this and by (\ref{eq 00}) and (\ref{eq 57}), we obtain that given the word $w$ in the form (\ref{eq 1}), the coefficients of the rewriting (\ref{eq 6}) of $w$ can be calculated in time
\begin{equation}
\label{eq 61}
 O( \left\| w \right\| Lg(\left\| w \right\|)) + O(log_2( \left\| w \right\| ) \left\| w \right\|^2) + O(log_2( \left\| w \right\| ) \left\| w \right\|^2)\leq
  O( \left\| w \right\|^2 Lg(\left\| w \right\|)).
\end{equation}

~\\
At this point, by Lemma \ref{lem 4444}, we got that the word problem in $G$ is reduced to the word problem in $K$, and according to (\ref{eq 61}), this reduction can be done in polynomial time. In turn, the word problem in $K$ can be reduced to the word problem in $H$.  Indeed,
$((b^{(i)})^{z^{\eta}})^{\xi}(z^{\nu})= (b^{(i)}(z^{\eta+\nu}))^{\xi} = (a^{(i)})^{\xi}$ if $\eta+\nu>0$ and it is $1$ if $\eta + \nu \leq 0$. Therefore,
 
the function $((b^{(i_1)})^{z^{\eta_1}})^{\xi_1} ((b^{(i_2)})^{z^{\eta_2}})^{\xi_2} \ldots  ((b^{(i_m)})^{z^{\eta_m}})^{\xi_m}(z^{\nu})$ is constant and equal to 
$ (a^{(i_1)})^{\xi_1}(a^{(i_2)})^{\xi_2} \ldots (a^{(i_m)})^{\xi_m}$ for all $ \nu \geq \eta_0$, and equal to $1$ for all $ \nu \leq -\eta_0$, where by $\eta_0$ we denote $max\{|\eta_1|, \ldots, |\eta_m|\}$. Thus, in order $f_{\mu}$ to be trivial, a necessary and sufficient condition is that $\eta=0$ and
\[
((b^{(i_1)})^{z^{\eta_1}})^{\xi_1} ((b^{(i_2)})^{z^{\eta_2}})^{\xi_2} \ldots  ((b^{(i_m)})^{z^{\eta_m}})^{\xi_m}(z^{\nu})
\]
is trivial for all $|\nu| \leq \eta_0$ and consequently, for all $|\nu| \leq \left\|w\right\|$. Thus, we reduced the word problem in $G$ to the word problem in $H$. Moreover, according to (\ref{eq 61}) this reduction is done in polynomial time, which particularly means that if the word problem in $H$ belongs to the class $P$ or $NP$, then the word problem in $G$ belongs to the same class.

Since above we described a procedure of reducing the word problem in $G$ to the word problem in $H$, we conclude that the word problem in $G$ is solvable, whenever the word problem in $H$ is solvable. Now let us finish the proof of the part $(i)$ of Theorem \ref{th 1}.

First, notice that since in (\ref{eq 6}) $0<i_j \leq log_2(\mu + \gamma_0)$, we have that 
\[
((b^{(i_1)})^{z^{\eta_1}})^{\xi_1} ((b^{(i_2)})^{z^{\eta_2}})^{\xi_2} \ldots  ((b^{(i_m)})^{z^{\eta_m}})^{\xi_m}(z^{\nu})
\]
\begin{equation}
\label{eq 1515}
=[((b^{(i_1)})^{z^{\eta_1}})^{\xi_1} (z^{\nu})] [((b^{(i_2)})^{z^{\eta_2}})^{\xi_2}(z^{\nu})] \ldots  [((b^{(i_m)})^{z^{\eta_m}})^{\xi_m}(z^{\nu})],
\end{equation}
regarded as a word in $S=\{ a^{(1)}, a^{(2)}, \ldots \}$, belongs to $S_{\left\lfloor log_2(\gamma_0+\mu)\right\rfloor}^* \subset S_{Lg(\left\|w\right\|)}^*$ (recall that by $S_n^*$ we denoted the set of words $\{ a^{(1)}, a^{(2)}, \ldots, a^{(n)} \}^*$). 
Hence, taking into account the inequality (\ref{eq 0000}), we have
\[
\mathfrak{L}(((b^{(i_1)})^{z^{\eta_1}})^{\xi_1} ((b^{(i_2)})^{z^{\eta_2}})^{\xi_2} \ldots  ((b^{(i_m)})^{z^{\eta_m}})^{\xi_m}(z^{\nu}) ) \leq mLg(Lg(\left\|w\right\|)) 
\]
\begin{equation}
\label{eq 888}
\leq \left\| w \right\|Lg(Lg(\left\|w\right\|))
\end{equation}
The last inequality means that triviality of the word 
\[
((b^{(i_1)})^{z^{\eta_1}})^{\xi_1} ((b^{(i_2)})^{z^{\eta_2}})^{\xi_2} \ldots  ((b^{(i_m)})^{z^{\eta_m}})^{\xi_m}(z^{\nu})
\in S_{Lg(\left\|w\right\|)}^*
\]
can be checked in time $O(f(\left\| w \right\|Lg(Lg(\left\|w\right\|)) )) $.
Thus, since $f_{\mu}$ is trivial iff the corresponding word (\ref{eq 1515}) is trivial in $H$ for all $|\nu| \leq \left\| w \right\|$, we conclude that given the rewriting  (\ref{eq 6}) (i.e., the coefficients in (\ref{eq 6})) of $f_{\mu}$, we can check its triviality in time
\begin{equation}
\label{eq 7}
O(2\left\|w\right\| f(\left\| w \right\|Lg(Lg(\left\|w\right\|)) )) =
O(\left\| w \right\| f( \left\| w \right\|Lg(Lg(\left\|w\right\|))),
\end{equation}
were $f$, as we mentioned in the statement of Theorem \ref{th 1}, is the time complexity function of the word problem in $H$ with respect to the encoding described in the introduction.
 
Combined this with (\ref{eq 61}), we conclude that triviality of $f_{\mu}$, when it is given in the form (\ref{eq 5}), can be checked in time  
\[
O( \left\| w \right\|^2 Lg(\left\| w \right\|) + \left\| w \right\| f( \left\| w \right\|Lg(Lg(\left\|w\right\|)))).
\] 
Hence, since $\gamma_0 \leq \left\| w \right\| $, triviality of (\ref{eq 4}), for all appropriate $\mu$-s, can be checked in time 
\[
O( \left\| w \right\|^3 Lg(\left\| w \right\|) + \left\| w \right\|^2 f( \left\| w \right\|Lg(Lg(\left\|w\right\|)))).
\]
~\\
Finally, combined this with the formula (\ref{eq new 1}) we conclude that one can check triviality of $w$ in time

\[
O(log_2(\left\| w \right\|)\left\| w \right\|^2)+O( \left\| w \right\|^3 Lg(\left\| w \right\|) + \left\| w \right\|^2 f( \left\| w \right\|Lg(Lg(\left\|w\right\|))))
\]
\begin{equation}
\label{eq 8}
=O( \left\| w \right\|^3 Lg(\left\| w \right\|) + \left\| w \right\|^2 f( \left\| w \right\|Lg(Lg(\left\|w\right\|)))).
\end{equation}

~\\
At this point we showed that the word problem for the group $G$ is solvable and also proved the part $(i)$ of Theorem \ref{th 1}.
~\\

Now let us solve the membership problem for the subgroup $H < G$, i.e., describe a procedure for determining whether the element in $G$ corresponding to $w$ belongs to the image of $H$ under the map $\phi$ or not.\\
~\\
The following lemma is a slightly modified version of Lemma \ref{lem 1}.

\begin{lemma}
\label{lem 3}

If the word (\ref{eq 1}), regarded as an element of $G$, belongs to $\phi(H)$, then $\sum_{j\in B_i}{\beta_j}=0$, for all $i= 1, \ldots, n$.

\end{lemma}

\begin{proof}
This is true, because, if $ w \in \phi(H)$, then $w(s^{\mu})=1$ for all $\mu \neq 1$.
The rest of the proof is the same as the proof of Lemma \ref{lem 1}.
\end{proof}

It is clear that $w\in \phi(H)$ iff in the rewriting (\ref{eq 2}) $\gamma=0$ and
\[(c^{s^{\gamma_1}})^{\beta_1} (c^{s^{\gamma_2}})^{\beta_2} \ldots  (c^{s^{\gamma_n}})^{\beta_n} (s^{\mu}) \in gp \langle [z,b^{(i)}] ~|~ i\in \mathbb{N} \rangle \] if $\mu=1$ and  it is $1$ if $\mu \neq 1$.

Combined this with Lemma \ref{lem 2} and Lemma \ref{lem 3} we get the following

\begin{lemma}
\label{lem 4}

$w \in \phi(H)$ iff $\gamma=0$, $\sum_{j\in B_i}{\beta_j}=0$ for ${i=1, \ldots, n}$,
\[(c^{s^{\gamma_1}})^{\beta_1} (c^{s^{\gamma_2}})^{\beta_2} \ldots  (c^{s^{\gamma_n}})^{\beta_n}(s^{\mu})=1\]
for $\mu \neq 1 \& |\mu|\leq 3\gamma_0$, ($\gamma_0=\max \{|\gamma_1|, |\gamma_2|, \ldots, |\gamma_n| \}$) and 
\[(c^{s^{\gamma_1}})^{\beta_1} (c^{s^{\gamma_2}})^{\beta_2} \ldots  (c^{s^{\gamma_n}})^{\beta_n}(s) \in gp\langle [z,b^{(i)}] ~|~ i \in \mathbb{N} \rangle.\]
\qed

\end{lemma}
~\\
So the membership problem for $\phi(H) < G$ is reduced to the membership problem for 
$ H<K $. (Since $ a^{(i)} \mapsto [z, b^{(i)}]$ induces an embedding of $H$ into $K$, we can regard $H$ as a subgroup of the group $K$. )

Considered the rewriting (\ref{eq 6}), let us use the following notation 
\[
\bar{b}= ((b^{(i_1)})^{z^{\eta_1}})^{\xi_1} ((b^{(i_2)})^{z^{\eta_2}})^{\xi_2} \ldots  ((b^{(i_m)})^{z^{\eta_m}})^{\xi_m}.
\]

\begin{lemma}
\label{lem 5555}
The element   in (\ref{eq 6}) belongs to $H$ iff
$\eta = 0$ and $\bar{b} (z^{\pm 1})=\bar{b} (z^{\pm 2}) = \ldots = \bar{b} (z^{\pm \eta_0}) =1 $.

\end{lemma}

\begin{proof}
Indeed, it follows from the fact that $\bar{b}(1) \in H$ and $\bar{b}(z^{\nu})=1$ for $\nu<-\eta_0$
and $\bar{b}(z^{\nu})=\bar{b}(z^{\eta_0})$ for $\nu>\eta_0$.

\end{proof}

So we see that the membership problem is reduced to the word problem in $K$. Hence, since the word problem in $K$ is solvable, we conclude that the membership problem is solvable as well. Moreover, just like in the case of the word problem described above, if the word problem in $H$ belongs to the class $P$ or $NP$, then the membership problem for $H<G$ belongs to the same class.\\
~\\
According to the formula (\ref{eq 8}), we can check the condition in Lemma \ref{lem 5555} in time 
$2|\eta| O( \left\| w \right\|^3 Lg(\left\| w \right\|) + \left\| w \right\|^2 f( \left\| w \right\|Lg(Lg(\left\|w\right\|))))$.
But $|\eta| \leq \left\| w \right\|$, which means that the membership problem for $w$ can be solved in time
\[
O( \left\| w \right\|^4 Lg(\left\| w \right\|) + \left\| w \right\|^3 f( \left\| w \right\|Lg(Lg(\left\|w\right\|)))).
\]

This completes the proof of the part $(ii)$ of Theorem \ref{th 1}.
~\\
~\\
Below we complete the proof of Theorem \ref{th 1} by showing the part (iii) and meanwhile, review the steps of the algorithms for the word and membership problems described above.

First, notice that the word $w$ in (\ref{eq 1}) can be encoded in a machine as a $(2n+1)$-tuple
$(\alpha_0, \beta_1, \ldots, \beta_n,\alpha_n)$. Since for each integer $k$, its binary presentation has length $O(log_2(|k|))$, we obtain that the $(2n+1)$-tuple $(\alpha_0, \beta_1, \ldots, \beta_n,\alpha_n)$ occupies only $O(\sum_{i=0}^n{log_2(|\alpha_i|)})\leq O(\left\| w \right\|)$ machine space or squares of the machine, as it is usual to call for Turing machines (or memory, in other words). At one of the steps in the above described word problem algorithm, we transferred from the form (\ref{eq 1}) to the form (\ref{eq 2}), which can be done using the formula (\ref{eq 0}). Since the formula (\ref{eq 0}) is only about using addition, it can be implemented so that the machine will not occupy more than $O(\left\| w \right\|)$ machine space at each point of the machine work (for example, at the $j$-th step we can keep in machine $\gamma_j$ instead of $(\alpha_0, \ldots, \alpha_{i-1})$ and at the $(j+1)$-st step, add $a_{j+1}$ to $\gamma_j$ and delete $a_{j+1}$ and so on. After we have got the form (\ref{eq 2}) in the machine, according to our word problem algorithm, it is time to check the condition (\ref{eq 99}), which will take only $O(\left\| w \right\|)$ machine space, since it involves only operation of addition of less than $\left\| w \right\|$ numbers. In the case this check shows negative result, our algorithm finishes its work, showing that $\left\|w\right\|$ is not trivial in $G$. Otherwise, our algorithm suggests that the next and final step should be checking the triviality of $(c^{s^{\gamma_1}})^{\beta_1}(c^{s^{\gamma_2}})^{\beta_2} \ldots (c^{s^{\gamma_n}})^{\beta_n}(s^{\mu})$ for all $|\mu| \leq \left\| w \right\|$. Since for different values of $\mu$ we can check this condition consecutively, not simultaneously, the space function for the whole procedure is the same as the space function for a procedure for a fixed $\mu$. We showed that $(c^{s^{\gamma_1}})^{\beta_1}(c^{s^{\gamma_2}})^{\beta_2} \ldots (c^{s^{\gamma_n}})^{\beta_n}(s^{\mu})$ for all $|\mu| \leq \left\| w \right\|$ can be presented in the form (\ref{eq 6}), where $\eta$ and $\eta_j$-s can be calculated by the formula (\ref{eq 99}) and $\zeta_i$-s in (\ref{eq 99}) can be calculated according to (\ref{eq 55}). Since again 
(\ref{eq 99}) involves only operation of addition and $m\leq \left\| w \right\|$, this procedure does not occupy more than $O(\left\| w \right\|)$ machine space. Implementation of the formula (\ref{eq 55}) does not use more than $O(\left\| w \right\|)$ machine space, because it is basically about calculating logarithms of numbers, which are less than $\left\| w \right\|$.
To summarize, up to this point we showed that passing to the form 
\[
z^{\zeta_0}(b^{(i_1)})^{\xi_1}z^{\zeta_1} (b^{(i_2)})^{\xi_2} \ldots  (b^{(i_m)})^{\xi_m}z^{\zeta_m}
\]
from (\ref{eq 6}), requires not more than $O(\left\| w \right\|)$ machine space.
It was shown above, that if $\left\| w \right\|$ is trivial in $G$, then 
\[
z^{\zeta_0}(b^{(i_1)})^{\xi_1}z^{\zeta_1} (b^{(i_2)})^{\xi_2} \ldots  (b^{(i_m)})^{\xi_m}z^{\zeta_m} (z^{\nu})
\]
is also trivial for all $|\nu| \leq \left\| w \right\|$. We can check the triviality of this word for different values of $\nu$ consecutively, so that it will use as much machine space as for checking triviality only for one fixed value of $\nu$.
Taken into account (\ref{eq 888}), checking the triviality for some fixed $\nu$ requires not more than $g(\left\| w \right\| Lg(Lg(\left\| w \right\| ))$ machine space. Thus, we obtained that the space function for the word problem algorithm is not larger than
\[
O(\left\| w \right\|) + O(g(\left\| w \right\| Lg(Lg(\left\| w \right\| )))= O(g(\left\| w \right\| Lg(Lg(\left\| w \right\| ))).
\]

Since triviality of the words in Lemma \ref{lem 5555} can be checked consecutively, the space complexity function of the membership problem algorithm coincides with the space complexity function of the algorithm for the word problem. 
~\\

\section{Proof of the corollaries}

Corollary \ref{corollary 1} follows from the general fact that if groups $A$ and $B$ are solvable of lengths $k$ and $l$ respectively, then $A Wr B$ is solvable as well and the length of solvability is at most $k+l$. This, and other properties of wreath products can be found in the book \cite{neumann-varieties}. \\
~\\\
In general, the estimation $l+2$, in Corollary \ref{corollary 1}, is optimal. For example, it was shown by Ph.Hall \cite{hall 6} that the additive group $\mathbb{Q}$ of rational numbers does not embed into a finitely generated metabelian group i.e., into a finitely generated group of solvability length at most 2.\\
~\\
Corollary \ref{corollary 2} follows from the fact that for torsion free groups $A$ and $B$, since $A Wr B/B^A \cong B$, we get that $A Wr B$ is torsion free as an extension of a torsion free group by a torsion free group. As a direct application of this fact, we have that if $H$ is torsion free then $H Wr Z$ is torsion free as well. Therefore, $ (H Wr Z) Wr S$ is torsion free. Thus, $G$ is torsion free as a non-trivial subgroup of a torsion free group. For the property of locally indicability exactly the same argument works.\\
~\\
Since the supports of the elements $b^{(i)}$ from the proof of Theorem \ref{th 1} are well-ordered subsets of $\mathbb{Z}$ (i.e., bounded from the left side), we have that the passive parts of the elements of the group $K = gp \left\langle z, b^{(i)} | i \in \mathbb{N}\right\rangle$ have well-ordered supports (they are well-ordered subsets of $\mathbb{Z}$). The same way, since $c \in K^S$ has a well-ordered support, we get that the supports of the passive parts of the elements of $G= gp \left\langle c, s \right\rangle$ are well-ordered subsets of $\mathbb{Z}$. Now, it becomes apparent, that since infinite cyclic groups are linearly ordered by their natural order, according to Lemma \ref{lem darbinyan mikaelian}, the group $K$ is linearly orderable. Therefore, again by Lemma \ref{lem darbinyan mikaelian}, the group $G < K Wr S$ is linearly orderable. Thus, Corollary \ref{corollary 3} is proved.
~\\
~\\
\textbf{Acknowledgement.}
I am grateful to Alexander Olshanskii for his encouragement to work on this paper, for his support and many valuable discussions.


\end{document}